\documentclass{amsart}
\allowdisplaybreaks[1]

\usepackage{amstext}
\usepackage{amsthm}
\usepackage{amssymb}

\newtheorem{thm}{Theorem}[section]
\newtheorem{lem}[thm]{Lemma}
\newtheorem{prop}[thm]{Proposition}

\newtheorem{conj}[thm]{Conjecture}
\theoremstyle{definition}
\newtheorem{defn}[thm]{Definition}
\newtheorem{ex}[thm]{Example}
\theoremstyle{remark}
\newtheorem{rem}[thm]{Remark}

\newcommand{\A}{\mathcal {A}}

\newcommand{\E}{\mathcal {E}}

\newcommand{\Q}{\mathbb{Q}}
\newcommand{\R}{\mathbb{R}}
\newcommand{\Z}{\mathbb{Z}}
\newcommand{\bk}{\mathbf{k}}
\newcommand{\hh}{\z_\A^{(2)}}
\newcommand{\hhs}{\zeta_{\mathcal{A}}^{(2),\star}}
\newcommand{\ZZ}{\mathcal {Z}}
\newcommand{\z}{\zeta}

\def\={\;=\;}

\begin{document}

\title{On finite multiple zeta values of level two}

\author{Masanobu Kaneko}
\address[Masanobu Kaneko]{Faculty of Mathematics, Kyushu University
 744, Motooka, Nishi-ku, Fukuoka, 819-0395, Japan}
\email{mkaneko@math.kyushu-u.ac.jp}

\author{Takuya Murakami}
\address[Takuya Murakami]{Sophia Fukuoka Junior and Senior High School,
1-10-10 Terukuni, Chuo-ku, Fukuoka, 810-0032, Japan}
\email{tak\_mrkm@icloud.com}

\author{Amane Yoshihara}
\address[Amane Yoshihara]{3-27-14-102, Nishi-Ohmiya, Nishi-ku, Saitama, 331-0078, Japan}
\email{amanettiyosshy@me.com}

\keywords{multiple zeta values, finite multiple zeta values, finite Euler sums}
\subjclass[2010]{Primary 11M32; Secondary 11A07}

\dedicatory{Dedicated to Professor Don Zagier on the occasion of his 70th birthday}

\maketitle
\begin{abstract}
We introduce and study a ``level two'' analogue of finite multiple zeta values.  We give conjectural bases 
of the space of finite Euler sums as well as that of usual finite multiple zeta values in terms of these newly defined elements.
A kind of ``parity result'' and certain sum formulas are also presented. 
\end{abstract}

\section{Definitions and conjectures}

The finite multiple zeta value $\z_\A(k_1,\ldots,k_r)$ is an element in the $\Q$-algebra $\A$ defined by
\[ \A\;:=
\prod_{p}\mathbb{Z}/p\mathbb{Z}\, \text{{\LARGE ${\textstyle \left. \middle/ \right.}$}}\!\bigoplus_{p}\mathbb{Z}/p\mathbb{Z} 
\=\{(a_{(p)})_p\,\vert\, a_{(p)}\in \Z/p\Z\}/\sim. \]
Here, $p$ runs over all prime numbers, and the relation $(a_{(p)})_p\sim (b_{(p)})_p$ means that the equality $a_{(p)}=b_{(p)}$ 
holds in $\Z/p\Z$ for all but a finite number of $p$.  
We often identify a representative $a=(a_{(p)})_p$ with the element in $\A$ that it defines.  Precisely, 
$\z_\A(k_1,\ldots,k_r)$ is defined as follows.  
 
\begin{defn} For a tuple of positive integers $(k_1,\ldots,k_r)$ (called an  index), define the ($\A$-) finite multiple zeta value $\z_\A(k_1,\ldots,k_r)\in\A$ by
\begin{equation}\label{def1}  \z_\A(k_1,\ldots, k_r)_{(p)}\=\sum_{0<m_1<\cdots <m_r<p}\frac1{m_1^{k_1}\cdots m_r^{k_r}}\ \bmod p.
\end{equation}
\end{defn}
This is a finite analogue of the usual multiple zeta value in $\R$:
\begin{equation}\label{mzv}   \z(k_1,\ldots, k_r)\=\sum_{0<m_1<\cdots <m_r}\frac1{m_1^{k_1}\cdots m_r^{k_r}}. \end{equation}
To ensure the convergence, we need the condition $k_r>1$ here, but for $\z_\A(k_1,\ldots,k_r)$, obviously we do not need
such a restriction.

In recent years, a vast amount of work has been done on the classical multiple zeta value \eqref{mzv} and its numerous variants and
generalizations including the finite multiple zeta value \eqref{def1}. A central conjecture concerning finite multiple zeta values
predicts a deep connection between finite and classical multiple zeta values (see \cite{Kan19, KZ} for the precise statement). 
For references to the extensive literature on the subject, one may refer to the book \cite{Zh} by Zhao
and the website \cite{Hweb} maintained by Hoffman.   

In this paper, we consider the following ``level two'' variant  $\z_\A^{(2)}(k_1,\ldots,k_r)$ of $\z_\A(k_1,\ldots,k_r)$.

\begin{defn} For an index $(k_1,\ldots,k_r)$, define the finite multiple zeta value of level two $\hh(k_1,\ldots,k_r)$ in $\A$ by 
\begin{equation}\label{def2}  \hh(k_1,\ldots, k_r)_{(p)}\=\sum_{0<m_1<\cdots <m_r<p/2}\frac1{m_1^{k_1}\cdots m_r^{k_r}}\ \bmod p.
\end{equation}
\end{defn}
The difference is that the summation extends up to $p/2$ instead of $p$.  We mention that various congruence properties 
of this sum for special indices was already 
considered in several literatures, for instance in Pilehrood-Pilehrood-Tauraso~\cite{PPT12, PPT}. See also \cite{Zh}.  

For later use, we note here that, first by putting $n_i=2m_i$ in the sum and then changing $n_i$ with $p-n_{i}$, 
we have two expressions 
\begin{align} \hh(k_1,\ldots, k_r)_{(p)}& \=2^{k_1+\cdots +k_r}
\sum_{0<n_1<\cdots <n_r<p\atop n_i:\, \text{even}}\frac1{n_1^{k_1}\cdots n_r^{k_r}}\ \bmod p  \label{even}  \\ 
&\= (-2)^{k_1+\cdots +k_r}
\sum_{0<n_r<\cdots <n_1<p\atop n_i:\, \text{odd}}\frac1{n_r^{k_r}\cdots n_1^{k_1}}\ \bmod p. \label{odd}
\end{align}
In particular, $\hh(k_1,\ldots,k_r)$ may be viewed as a finite analogue of Hoffman's ``t-value'' \cite{Hof19}, up to a constant multiple.  
We further note that, if we write the first sum as 
\[  \sum_{0<n_1<\cdots <n_r<p\atop n_i:\, \text{even}}\frac1{n_1^{k_1}\cdots n_r^{k_r}}
=2^{-r}\sum_{0<n_1<\cdots <n_r<p}\frac{(1+(-1)^{n_1})\cdots (1+(-1)^{n_r})}{n_1^{k_1}\cdots n_r^{k_r}}, \]
we see that $\hh(k_1,\ldots,k_r)$ can be written as a $\Q$-linear combination of ``finite Euler sums,'' as studied for instance
in Zhao~\cite{Zh11, Zh}. 

We introduce three $\Q$-subspaces of $\A$ spanned by the usual finite multiple zeta values, our
level-two analogues, and the finite Euler sums.

\begin{defn} For each integer $k\ge0$, define the $\Q$-vector spaces $\ZZ_{\A,k}^{(1)},\, \ZZ_{\A,k}^{(2)}$,
and $\E_k$ in $\A$ by $\ZZ_{\A,0}^{(1)}=\ZZ_{\A,0}^{(2)}=\E_0=\Q$ and  
\[ \ZZ_{\A,k}^{(1)}:=\sum_{k_1+\cdots+k_r=k\atop r\ge1,\,\forall k_i\ge1}\Q\cdot\z_\A(k_1,\ldots,k_r) \quad(k\ge1),\]
\[ \ZZ_{\A,k}^{(2)}:=\sum_{k_1+\cdots+k_r=k\atop r\ge1,\,\forall k_i\ge1}\Q\cdot\z_\A^{(2)}(k_1,\ldots,k_r) \quad(k\ge1),\]
and \[ \E_k:= \text{$\Q$-span of all finite Euler sums of weight $k$}, \]
namely, all elements in $\A$ of the form 
\[ \left(\sum_{0<m_1<\cdots <m_r<p}\frac{(\pm1)^{m_1}\cdots (\pm1)^{m_r}}{m_1^{k_1}\cdots m_r^{k_r}}\ \bmod p\right)_p \]
with $k_1+\cdots +k_r=k\ (r\ge 1, \forall k_i\ge1)$ and with all possible signs in the numerator.
 Further, we set 
\[ \ZZ_\A^{(1)}\;:=\;\sum_{k=0}^\infty \ZZ_{\A,k}^{(1)},\quad \ZZ_\A^{(2)}\;:=\;\sum_{k=0}^\infty \ZZ_{\A,k}^{(2)},\quad 
\E\;:=\;\sum_{k=0}^\infty \E_k. \]
 \end{defn}

\begin{prop} \label{prop1}
The space $\ZZ_\A^{(2)}$ is a  $\Q$-subalgebra of $\E$.
\end{prop}
\begin{proof}  That the space $\ZZ_\A^{(2)}$ is contained in $\E$ 
has already been remarked above.  And that  $\ZZ_\A^{(2)}$ is closed under
multiplication is seen by the fact that the standard harmonic (or stuffle) product rule
applies also to the defining sum of $\hh(k_1,\ldots,k_r)$ in \eqref{def2}.  
($\E$ is a $\Q$-algebra by the same reasoning.) \end{proof}

Based on an evidence supported by numerical experiments, we propose the following conjecture.

\begin{conj}\label{conj1}
i)  $\ZZ_\A^{(2)}=\E$.

ii)  The set $\{ \hh(k_1,\ldots, k_r)\,\mid\, r\ge1,\,\forall k_i: \text{odd} \ge1 \}$ forms a linear basis of $\ZZ_\A^{(2)}$.
\end{conj}

\begin{rem}  The conjectural dimension (as a $\Q$-vector space) of $\E_k\ (k\ge1)$ is given by the Fibonacci number $F_k$
($=F_{k-1}+F_{k-2},\ F_1=F_2=1$) (cf. \cite[\S8.6.3]{Zh}).  The cardinality of the set in ii) above is easily seen to be equal to $F_k$.
Also note that the number of $\hh(k_1,\ldots, k_r)$ of weight $k$ is $2^{k-1}$ which is much smaller than that of 
finite Euler sums of weight $k$, namely $2\cdot 3^{k-1}$.
\end{rem}

\begin{prop} \label{prop2}
The space $\ZZ_\A^{(1)}$ of ordinary finite multiple zeta values is contained in  $\ZZ_\A^{(2)}$; we have the inclusions
$\ZZ_\A^{(1)}\subset \ZZ_\A^{(2)}\subset \E$.
\end{prop}

\begin{proof}   This can be seen from the identity
\begin{equation}\label{key} 
\z_\A(k_1,\ldots,k_r)\=\sum_{i=0}^r (-1)^{k_{i+1}+\cdots +k_r} \hh(k_1,\ldots,k_i)\hh(k_r,\ldots,k_{i+1}),
\end{equation}
where we set $\hh(\emptyset)=1$, together with Proposition~\ref{prop1} (that $\ZZ_\A^{(2)}$ is closed under multiplication).  
This identity, which is also useful later, is a consequence of the following division of the sum
\[ \sum_{0<m_1<\cdots <m_r<p}\= \sum_{i=0}^r\ \sum_{0<m_1<\cdots<m_i<p/2<m_{i+1}<\cdots <m_r<p} \]
 in the definition and the change $m_j\to p-m_j$ for $j=i+1,\ldots,r$.
\end{proof}

\begin{rem}  This is reminiscent of the definition of the ``symmetric multiple zeta values,'' a conjectural real counterpart
of finite multiple zeta values. See \cite{Kan19, KZ} for the details on this.  
\end{rem}

Also from the numerical experiments, we surmise 

\begin{conj}\label{conj2}
i) If all $k_i$ are greater than 1,  each $\hh(k_1,\ldots,k_r)$ is in $\ZZ_\A^{(1)}$.

ii)  The set $\{ \hh(k_1,\ldots, k_r)\,\mid\, r\ge1,\,\forall k_i: \text{odd} \ge3\}$ constitutes a basis of $\ZZ_\A^{(1)}$.
\end{conj}

\begin{rem}  The conjectural dimension of $\ZZ_\A^{(1)}$ is given by the sequence $d_{k-3}$
defined recursively by $d_k=d_{k-2}+d_{k-3},\ d_0=1, d_1=0, d_2=1$.  (cf. \cite{Kan19, Zh}).  
The cardinality of the set in ii) above  equals  $d_{k-3}$.
\end{rem}

\section{Examples in low depths and a parity result}

First  we define two specific elements $L(2)$ and $Z(k)\ (k\ge2)$ in $\A$ as
\[ L(2)_{(p)} := \frac{2^{p-1}-1}{p} \bmod p\quad\text{and}\quad Z(k)_{(p)}:=\frac{B_{p-k}}{k} \bmod p, \]
where $B_n$ denotes the $n$th Bernoulli number.  Note that, by the definition of $\A$, we may ignore
possible (finitely many) $p$'s such that the right-hand sides are not well defined.  These elements are respectively 
a natural analogue of $\log2$ and the conjectural ``true'' analogue of $\zeta(k)\bmod \pi^2$ in $\A$.  We refer the reader to \cite{KZ}
for more details on these.  We first recall known formulas for depth (the length of the index) less than or equal to 2 (\cite[Th.~5.2]{Sun00}, \cite[Lem.~1]{PPT12}).
We give proofs for the convenience of the reader.

\begin{prop}[\cite{Sun00}, \cite{PPT12}] \label{lowdep}
i) $\hh(1)=-2L(2)$ and $\hh(k)=(2-2^{k})Z(k)$ for $k\ge2$.

In particular, $\hh(k)=0$ if $k$ is even.

ii) If $k_{1}+k_{2}$ is odd,
\[
\hh(k_{1},k_{2})=\frac{1}{2}\left\{ (-1)^{k_{2}}\binom{k_{1}+k_{2}}{k_{2}}+2^{k_{1}+k_{2}}-2\right\} Z\left(k_{1}+k_{2}\right).
\]

\end{prop}
\begin{proof}
i)  From the computation using the binomial formula, we have
\[ 2L(2)_{(p)} \= \frac{(1+1)^p-2}p \bmod p \= \sum_{i=1}^{p-1}\frac{(-1)^{i-1}}i \bmod p. \]
This is equal to  $-\hh(1)$ because
\[ \sum_{0<i<p\atop i:\text{even}}\frac{1}i \, \bmod p \= \frac12 \hh(1)_{(p)} \quad \text{and} \quad
\sum_{0<i<p\atop i:\text{odd}}\frac{1}i \, \bmod p \= -\frac12 \hh(1)_{(p)}, \]  as seen from \eqref{even} and \eqref{odd}.

For the second equality, we use the Seki-Bernoulli formula for sum of powers (cf. \cite{AIK14}). 
We start with 
\[ \hh(k)_{(p)}\= \sum_{0<m<p/2}\frac1{m^k}\, \bmod p \= \sum_{0<m<p/2} m^{p-1-k} \,\bmod p. \]
The last sum, for large enough $p$, is equal to $\bigl(B_{p-k}\bigl(\frac{p+1}2\bigr)-B_{p-k}(1)\bigr)/(p-k)$,
where $B_n(x)$ denotes the Bernoulli polynomial (cf. \cite[Rem. 4.10]{AIK14}). From the formula
$B_n(1/2)=(2^{1-n}-1)B_n$ (easily derived from the distribution relation \cite[Prop. 4.9 (7)]{AIK14} for the case $k=2$),
we see that this quantity is congruent modulo $p$ to $(2-2^k)B_{p-k}/k$, the result follows. When $k$ is even, $B_{p-k}=0$
for almost all $p$ and thus $\hh(k)=0$.  

We prove ii) in Example~\ref{ex} after we present a general identity \eqref{parity}.
\end{proof}
The equality in ii) above may be viewed as an analogue of the ``parity result'' (\cite{IKZ06, Tsu04}) in the case of depth 2.  
In the next proposition we present a general identity from which one can obtain (a kind of)
a general parity result.  To state the proposition, we introduce the ``star'' variant
$\hhs(k_{1},\ldots,k_{r})$ defined similarly as $\hh(k_{1},\ldots,k_{r})$ but the summation is over 
$0<m_{1}\leq\cdots\leq m_{r}<p/2$ rather than with the strict inequalities.

\begin{prop}  For any $r\ge1$ and $k_i\ge1$, we have
\begin{equation}\label{parity}
\hh(k_{1},\dots,k_{r})=(-1)^{r+k_{1}+\cdots+k_{r}}\sum_{i=0}^{r}(-1)^{i}\zeta_{\mathcal{A}}(k_{i},\dots,k_{1})\hhs(k_{i+1},\dots,k_{r}).
\end{equation}
\end{prop}

\begin{proof}
We use \eqref{key} and the ``antipode identity'' (a consequence of the harmonic algebra structure, see \cite{Hof15})
\begin{equation}
\sum_{j=0}^{r}(-1)^{j}\hh(k_{j},\dots,k_{1})\hhs(k_{j+1},\dots,k_{r})=\delta_{r,0},\label{key2}
\end{equation}
where $\delta_{r,0}$ is 0 if $r>0$ and 1 if $r=0$.  By using these and the reversal formula
$\zeta_{\mathcal{A}}(k_{i},\dots,k_{1}) = (-1)^{k_{1}+\cdots+k_{i}}\zeta_{\mathcal{A}}(k_{1},\dots,k_{i})$, 
we start with the sum on the right and proceed as
\begin{align*}
 & \sum_{i=0}^{r}(-1)^{i}\zeta_{\mathcal{A}}(k_{i},\dots,k_{1})\hhs(k_{i+1},\dots,k_{r})\\
= & \sum_{i=0}^{r}(-1)^{i+k_{1}+\cdots+k_{i}}\zeta_{\mathcal{A}}(k_{1},\dots,k_{i})\hhs(k_{i+1},\dots,k_{r})\\
= & \sum_{i=0}^{r}(-1)^{i+k_{1}+\cdots+k_{i}}\sum_{j=0}^{i}(-1)^{k_{j+1}+\cdots+k_{i}}\hh(k_{1},\dots,k_{j})\hh(k_{i},\dots,k_{j+1})\hhs(k_{i+1},\dots,k_{r})\\
= & \sum_{j=0}^{r}\sum_{i=j}^{r}(-1)^{i+k_{1}+\cdots+k_{j}}\hh(k_{i},\dots,k_{j+1})\hhs(k_{i+1},\dots,k_{r})\hh(k_{1},\dots,k_{j})\\
= & (-1)^{r+k_{1}+\cdots+k_{r}}\hh(k_{1},\dots,k_{r}).
\end{align*}
\end{proof}

\begin{rem}
If $k_{1}+\cdots+k_{r}\not\equiv r \bmod2$, we obtain from \eqref{parity}
\[
\hh(k_{1},\dots,k_{r})+\hhs(k_{1},\dots,k_{r})=-\sum_{i=1}^{r}(-1)^{i}\zeta_{\mathcal{A}}(k_{i},\dots,k_{1})\hhs(k_{i+1},\dots,k_{r}).
\]
Noting $\hh(k_{1},\dots,k_{r})+\hhs(k_{1},\dots,k_{r})=2\hh(k_{1},\dots,k_{r})+ \sum\hh(\text{lower depth})$, we conclude 
that if the weight and the depth have a different parity, $\hh(k_{1},\dots,k_{r})$ is written as a sum of products of  $\hh$'s of lower depths
and $\z_\A(k_1,\ldots,k_r)$. If we view the depth of $\z_\A(k_1,\ldots,k_r)$ as $r-1$ (this is reasonable in light of our ``main conjecture''
in \cite{KZ}), this gives a kind of parity result for finite multiple zeta values of level~2, although it has the term $\z_\A(k_1,\ldots,k_r)$ of level one.
\end{rem}

\begin{ex}\label{ex} i)  When $r=2$, the identity \eqref{parity} becomes
\[ \hh(k_1,k_2)\= (-1)^{k_1+k_2}\bigl(\hhs(k_1,k_2)-\z_\A(k_1)\hhs(k_2)+\z_\A(k_2,k_1)\bigr),\]
and this is equal to 
\[  (-1)^{k_1+k_2}\bigl(\hh(k_1,k_2)+\hh(k_1+k_2)+\z_\A(k_2,k_1)\bigr) \]
because $\z_\A(k_1)=0$ for any $k_1$ (see \cite{Hof15, Zh08}, also \cite{Kan19, KZ}). 
If $k_1+k_2$ is odd, we see from this that
\[ \hh(k_1,k_2)\= -\frac12\bigl(\hh(k_1+k_2)+\z_\A(k_2,k_1)\bigr). \]
Then Proposition~\ref{lowdep} ii) follows from Proposition~\ref{lowdep} i) and a formula for 
$\z_\A(k_2,k_1)$ in \cite{Hof15, Zh08, Kan19, KZ}.

ii)  The case $r=3$ of \eqref{parity} reads (we set $k=k_{1}+k_{2}+k_{3}$ and use $\z_\A(k_1)=0$)
\[ \hh(k_{1},k_{2},k_{3}) = (-1)^{k-1}\bigl(\hhs(k_1,k_2,k_3)+\z_\A(k_2,k_1)\hhs(k_3)-\z_\A(k_3,k_2,k_1)\bigr). \]
If $k$ is even, we have, by writing $\hhs(k_1,k_2,k_3)$ as a sum of $\hh$ in a usual way, 
\begin{align*}
&\hh(k_{1},k_{2},k_{3})\\
&=\frac{1}{2}\bigl( \zeta_{\mathcal{A}}(k_{1},k_{2},k_{3})-\hh(k_{1}+k_{2},k_{3})-\hh(k_{1},k_{2}+k_{3})+\zeta_{\mathcal{A}}(k_{1},k_{2})\hh(k_{3})\bigr).
\end{align*}
\end{ex}

\section{Sum formulas}

In this section, we present various sum formulas. First, we establish formulas for  
\[ S(k,r):=  \sum_{k_1+\cdots+k_r=k}\hh(k_1,\ldots,k_r) \quad\text{and}\quad 
S_{1}(k,r):=  \sum_{k_1+\cdots+k_r=k\atop \forall k_i\ge2}\hh(k_1,\ldots,k_r) , \]
writing them as linear combinations of conjectural basis elements given in Cojectures~\ref{conj1} and \ref{conj2}
respectively.  
Set
\[  B(k,r):=  \sum_{k_1+\cdots+k_r=k\atop \forall k_i:\text{odd} \ge1}\hh(k_1,\ldots,k_r) \quad\text{and}\quad 
B_{1}(k,r):=  \sum_{k_1+\cdots+k_r=k\atop \forall k_i:\text{odd}\ge3}\hh(k_1,\ldots,k_r).\]

\begin{thm}\label{thm1} For $1\le r\le k$, we have

i)
\[
S(k,r)=(-1)^{k+r}\sum_{\substack{1\leq i\leq r\\
i\equiv k\bmod 2
}
}\binom{\frac{k-i}{2}}{r-i}B(k,i)
\]
and

ii)
\[
S_{1}(k,r)=(-1)^{k+r}\sum_{\substack{1\leq i\leq r\\
i\equiv k\bmod 2
}
}\binom{\frac{k-3i}{2}}{r-i}B_{1}(k,i).
\]
\end{thm}

The following is the key lemma to prove Theorem~\ref{thm1}. We consider the $\Q$-vector space 
spanned by formal symbols $[\bk]$ for each index $\bk$, equipped with the algebra structure 
given by the harmonic (stuffle) product $*$. 
For example, $[(2)]*[(3)]=[(2,3)]+[(3,2)]+[(5)]$.  Denote by $\mathcal{R}$ this $\Q$-algebra. This is 
isomorphic to Hoffman's harmonic algebra $\mathfrak{h}^1$ (\cite{Hof97}). For more
details, we refer \cite{Hof97, Kan19}.

\begin{lem} For $1\le r\le k$ and $0\le a\le r$, set 
\[ g(k,r,a):=\sum_{k_1+\cdots+k_r=k\atop\#\,\text{of even }k_i\,=\,a}[(k_1,\ldots,k_r)] \quad\text{and}\quad 
g_1(k,r,a):=\sum_{k_1+\cdots+k_r=k\atop{\#\,\text{of even }k_i\,=\,a\atop \forall k_i\ge2}}[(k_1,\ldots,k_r)]. \]
Then we have the identity

i)
\[
\sum_{i=1}^{\frac{k-r-a}{2}}[(2i)]*g(k-2i,r,a)=\frac{k-r-a}{2}g(k,r,a)+(a+1)g(k,r+1,a+1)
\]
in $\mathcal{R}$ and also

ii)
\[
\sum_{i=1}^{\frac{k-3r+a}{2}}[(2i)]*g_{1}(k-2i,r,a)=\frac{k-3r+a}{2}g_{1}(k,r,a)+(a+1)g_{1}(k,r+1,a+1).
\]
\end{lem}

\begin{proof} i)  If we compute the harmonic product $[(2i)]*g(k-2i,r,a)$, each resulting term has weight $k$, 
depth either $r$ or $r+1$ and the number of even entries $a$ or $a+1$ respectively, i.e., a term appearing 
either in $g(k,r,a)$ or in $g(k,r+1,a+1)$.  For a given $[(k_1,\ldots,k_r)]$ in $g(k,r,a)$, the number of possible
combinations of $i$ and a term in $g(k-2i,r,a)$ which give $[(k_1,\ldots,k_r)]$ in their harmonic product  
is $(k-r-a)/2$, because the depth $r$ and the number $a$ of 
even entries are the same, so the choice is the entry $k_j$ in $[(k_1,\ldots,k_r)]$ which is larger than 2 and
the amount $2i$ such that $(k_1,\ldots,k_j-2i,\ldots,k_r)]$ is still an index ($k_j-2i>0)$.  Such a pair $(k_j,i)$
is unique, and the total number is $(k-(r-a)-2a)/2=(k-r-a)/2$ (there are $r-a$ odd entries and $a$ even entries in $(k_1,\ldots,k_r)$,
and $k-(r-a)-2a$ is the `excess' for possible subtraction of $2i$).
The term $[(k_1,\ldots,k_r)]$ in $g(k,r+1,a+1)$ comes from $[(2i)]*g(k-2i,r,a)$ by inserting $2i$
to a term in $g(k-2i,r,a)$, and so the choice is $a+1$.

The formula in ii) is proved similarly, just by noting the condition that all entries are greater than or equal to 2.
\end{proof}

\noindent{\it Proof of Theorem~\ref{thm1}.}
i)  Set 
\[ S(k,r,a):=\sum_{k_1+\cdots+k_r=k\atop\#\,\text{of even }k_i\,=\,a}\hh(k_1,\ldots,k_r). \]
Since $\hh(2i)=0$ (Proposition~\ref{lowdep} i)) and  $\hh$ obeys the harmonic product rule, one concludes from 
i) of the above lemma that 
\[
\frac{k-r-a}{2}S(k,r,a)+(a+1)S(k,r+1,a+1)\=0.
\]
From this, and noting $S(k,r,0)=B(k,r)$, we have
\begin{align*}
S(k,r,a) & =\left(-\frac{k-r-a+2}{2a}\right)S(k,r-1,a-1)\\
 & =\left(-\frac{k-r-a+2}{2a}\right)\left(-\frac{k-r-a+4}{2(a-1)}\right)S(k,r-2,a-2)\\
 & =\dots\\
 & =\left(-\frac{k-r-a+2}{2a}\right)\left(-\frac{k-r-a+4}{2(a-1)}\right)\cdots\left(-\frac{k-r-a+2a}{2}\right)S(k,r-a,0)\\
  & =(-1)^{a}\binom{\frac{k-r+a}{2}}{a}B(k,r-a).
\end{align*}
Summing up, we obtain i). The proof of ii) is the same and is omitted.  \qed \\

In an attempt to find a sum formula which is more close in form to the classical one (cf. \cite[Ch.~5]{Zh}), 
we discovered experimentally several formulas,  some we could prove and the other conjectural. Some of them 
look strange.
Since we think there still is much to be discovered and our understanding is not mature yet, 
we mention only several of them, give just sketches of proofs, postponing the detailed study in a future publication \cite{Mur22} by the second-named
author.  

\begin{thm}  For $1\le r\le k$ and a fixed $i$, we have
\[
\sum_{k_1+\cdots+k_r=k\atop
k_{i}:\text{odd and } \forall k_{j}:\text{even } (j\neq i)
} \hh(k_1,\ldots,k_r)\= c\,  \hh(k)\quad\text{for some rational constant }c.
\]
\end{thm}
\begin{proof}  A special case 
\[
\hh(\underbrace{2,\dots2}_{i-1},1,\underbrace{2,\dots,2}_{r-i})=\frac{(-1)^{r-1}}{2^{2r-2}}\binom{2r-1}{2i-1}\hh(2r-1)
\]
is proved in \cite[Th.~5.4]{PPT}, and we may use this and the harmonic product to establish the theorem.  
We are not able to obtain a general closed formula of the constant $c$. 
\end{proof}
The result \cite[Th.~5.3]{PPT} is also a special case and there the constant is explicit.
The above theorem looks similar to the classical sum formula for multiple zeta values. There are several 
variants like restricted sum formulas or weighted sum formulas (see for instance \cite{Zh}). But the next formulas 
look rather strange and seem similar to none of these.  We introduce one notation.

\begin{defn} For an index $(k_{1},\dots,k_{r})$ of weight $k=k_1+\cdots+k_r$, put
\[
C(k_{1},\dots,k_{r}):=\sum_{j=1}^{r-1}(-1)^{k_{1}+\cdots+k_{j}}\binom{k}{k_{1}+\cdots+k_{j}}.
\]
\end{defn}

Let $\Sigma_n$ be the set of permutations of $\{1,\ldots,n\}$ (the symmetric group of order $n$).
The following is a theorem for the usual (level one) finite multiple zeta values.

\begin{thm}  For a non-empty index $(k_{1},\dots,k_{r})$ of depth $r$ and weight $k$, we have
\begin{align*}
&\sum_{\sigma\in\Sigma_{r}}\left(r+1-2\sigma^{-1}(r)\right)\zeta_{\mathcal{A}}(k_{\sigma(1)},\dots,k_{\sigma(r)})\\
&\=(-1)^{r}\,2\!\sum_{\tau\in\Sigma_{r-1}}C(k_{\tau(1)},\ldots,k_{\tau(r-1)},k_{r})\cdot Z(k).
\end{align*}
\end{thm}
And the next is a level-two counterpart.

\begin{thm}
For a non-empty index $(k_{1},\dots,k_{r})$ of depth $r$ and weight $k$ with $k_i$ even for all
$1\le i\le r-1$ and $k_r$ odd, we have
\begin{align*}
&\sum_{\sigma\in\Sigma_{r}}\left(r+1-2\sigma^{-1}(r)\right)\hh(k_{\sigma(1)},\dots,k_{\sigma(r)})\\
&\=(-1)^{r}\sum_{\tau\in\Sigma_{r-1}}C(k_{\tau(1)},\ldots,k_{\tau(r-1)},k_{r})
\cdot Z(k).
\end{align*}
\end{thm}

The proofs of both theorems rely on the following lemma.

\begin{lem}  For an index $(k_1,\ldots,k_r)$, set 
\[ R(k_1,\ldots,k_r)=\sum_{\sigma\in\Sigma_r}(r+1-2\sigma^{-1}(r))[(k_{\sigma(1)},\ldots,k_{\sigma(r)})]. \]
Then, we have the identity in $\mathcal{R}$
\begin{align*}
&\sum_{i=1}^{r-1}[(k_{i})]*R(k_{1},\dots,\check{k_{i}},\dots,k_{r}) \\
& =(r-2)R(k_{1},\dots,k_{r})+\sum_{1\leq i\leq r-1}R(k_{1},\dots,\check{k_{i}},\dots,k_{r-1},k_{i}+k_{r})\\
 & +2\!\!\sum_{1\leq i<j\leq r-1}R(k_{i}+k_{j},k_{1},\dots,\check{k_{i}},\dots,\check{k_{j}}\dots,k_{r-1},k_{r}),
\end{align*}
where the wedge $\check{k_{i}}$ means $k_i$ is omitted.
\end{lem}
The proof of the lemma is done basically by comparing coefficients of terms on both sides, though this is a bit tedious.
And the proofs of theorems are by induction on depths, starting point being explicit formulas in the case of depth 2
(Proposition~\ref{lowdep} for level 2 and \cite{Hof15, Zh08}, \cite[Ex. 7.4]{Kan19} for level 1).
The detailed discussion will be given in \cite{Mur22}.

We end this paper by  a conjecture, which may be viewed as a variant of the weighted sum formula but
also strange in form.

\begin{conj}  For $r\ge1$ and $a\ (0\leq a\leq r)$, one has
\[
\sum_{\forall k_{i}\in\{1,2\}\atop \#\{i\mid k_i=2\}=a}\bigl( (-1)^{\#\{i\mid i:odd,k_{i}=2\}}2^{a}-1\bigr)\hh(k_{1},\dots,k_{r})\=0.\quad (\text{The weight is $r+a$}.)
\]
\end{conj}

\section*{Acknowledgement}
The authors, in particular M.~K., thank Don Zagier with great admiration for his 
enthusiasm and commitment to mathematics which have always been our source of inspiration over several decades.
This work was supported by JSPS KAKENHI Grant Numbers JP16H06336 and JP21H04430.


\end{document}